\documentclass[12pt]{amsart}%
\usepackage{amsmath}
\usepackage{graphicx}
\usepackage{amscd}
\usepackage{geometry}
\usepackage{amsfonts}
\usepackage{amssymb}%
\usepackage{amsthm}
\usepackage{verbatim}

\usepackage{tikz}

\usetikzlibrary{arrows}

\newtheorem{theorem}{Theorem}[section]
\theoremstyle{plain}

\newtheorem{corollary}[theorem]{Corollary}

\newtheorem{proposition}[theorem]{Proposition}

\newtheorem{lemma}[theorem]{Lemma}

\theoremstyle{definition}
\newtheorem{definition}[theorem]{Definition}
\newtheorem{definitions}[theorem]{Definitions}
\newtheorem{example}{Example}

\newtheorem{notation}{Notation}

\newtheorem{prop}[theorem]{Proposition}

\usetikzlibrary{arrows}

\numberwithin{equation}{section}

\let\oldtocsection=\tocsection
\let\oldtocsubsection=\tocsubsection
\let\oldtocsubsubsection=\tocsubsubsection
\renewcommand{\tocsection}[2]{\hspace{0em}\oldtocsection{#1}{#2}}
\renewcommand{\tocsubsection}[2]{\hspace{2em}\oldtocsubsection{#1}{#2}}
\renewcommand{\tocsubsubsection}[2]{\hspace{4em}\oldtocsubsubsection{#1}{#2}}

\usepackage{graphicx}

\usepackage{amsmath}

\usepackage{amsfonts}

\usepackage{amsthm}

\usepackage{amssymb}

\usepackage{latexsym}
\usepackage{fancyhdr}

\usepackage{wrapfig}

\usepackage{enumerate}
\usepackage{soul}
\usepackage{enumitem} 

\theoremstyle{definition}

\newcommand{\intr}{\text{int}}
\newcommand{\fr}{\text{fr}}
\newcommand{\inr}{\text{in}}
\newcommand{\comp}{\text{Comp}}

\newcommand{\mcal}{\mathcal}
\newcommand{\id}{\text{id}}
\newcommand{\cl}{\text{cl}}

\begin{document}
\title[Generalized dunce hats are not splittable]{Generalized dunce hats are not splittable}
\author{Fredric Ancel and Pete Sparks}

\date{?}
%
%
%
%
%
%
%
%
%
%
%
%

\begin{abstract} A \emph{generalized dunce hat} is a 2-dimensional polyhedron created by attaching the boundary of a disk $\Delta$ to a circle $J$ via a map $f:\partial \Delta \to J$ with the property that there is a point $v \in J$ such that $f^{-1}(\{v\})$ is a finite set containing at least 3 points and $f$ maps each component of $\partial \Delta - f^{-1}(\{v\})$ homeomorphically onto $J - \{v\}.$  \textbf{Theorem:}  No generalized dunce hat is the union of two proper subpolyhedra that each have finite first homology groups.  This result undermines a strategy for proving that the interior of the Mazur compact contractible 4-manifold M is \emph{splittable in the sense of Gabai} (i.e., $\intr(M) = U \cup V$ where $U,$ $V$ and $U \cap V$ are each homeomorphic to Euclidean 4-space).

\end{abstract}
\maketitle

\section{Introduction}
An open (non-compact boundaryless) $n$-manifold is \emph{splittable in the sense of Gabai} (or, more briefly, \emph{splittable}) if it is the union of two open subsets $U$ and $V$ such that $U,\  V$ and $U \cap V$ are homeomorphic to Euclidean $n$-space.  It is easily seen that Euclidean $n$-space is splittable.  What is not obvious is whether there exist any other splittable contractible open manifolds, and whether there exist any non-splittable contractible open manifolds.

	In 2009 David Gabai surprised the geometric topology community by showing that the \emph{Whitehead 3-manifold} is splittable [G].  The Whitehead 3-manifold is one of the fundamental examples in 3-manifold topology; it is a contractible open 3-manifold that is not homeomorphic to Euclidean 3-space.  (It was discovered by J. H. C. Whitehead in 1935 as a fatal obstruction to his own proposed proof of the 3-dimensional Poincare Conjecture. [W])  Gabai's observation inspired alternative proofs and generalizations by other topologists: in [GRW] it is proved that there exist uncountably many topologically distinct splittable contractible open 3-manifolds and that there exist uncountably many topologically distinct non-splittable contractible open 3-manifolds.  Pete Sparks in his 2014 Ph.D. thesis [S] extended the study of the spittability phenomenon to dimension 4 by constructing an uncountable family of topologically distinct splittable contractible open 4-manifolds.  The question of whether there exists a non-splittable contractible open 4-manifold remains unresolved.

	The Mazur 4-manifold [M] is a compact contractible 4-manifold with boundary whose interior is not homeomorphic to Euclidean 4-space.  It is a fundamental object in the study of 4-dimensional manifolds.  It is not known whether the interior of the Mazur 4-manifold is splittable.  One strategy for attacking this question relies on the fact that the Mazur 4-manifold has a \emph{spine} that is a \emph{dunce hat}.  
In general, a subpolyhedron $P$ of a compact piecewise linear manifold $M$ is called a \emph{spine} of $M$ if $M$ collapses to $P$.  The \emph{dunce hat} is a 2-dimensional polyhedron that is contractible but not collapsible, and according to [Z], the Mazur 4-manifold has a spine which is a dunce hat.  In general, if a compact piecewise linear manifold has a spine that can be expressed as the union of two collapsible proper subpolyhedra that intersect in a collapsible subpolyhedron, then these subpolyhedra can be thickened to give a splitting of the interior of the manifold.  (See Proposition 3.1.9 of [S].)  Thus, a potential strategy for proving the splittability of the interior of the Mazur 4-manifold would be to exhibit the dunce hat as the union of two collapsible subpolyhedron that intersect in a collapsible subpolyhedron.  The Theorem proved in this paper implies that this strategy can't work, because the dunce hat can't be divided into two collapsible subpolyhedra.  So the interior of the Mazur 4-manifold remains a potential candidate for a non-splittable contractible open 4-manifold.


%

	One might hope to prove that because the Mazur 4-manifold has a spine that can't be split into two collapsible subpolyhdra, then its interior is not splittable.  However, there are obstacles to giving such a proof because the relationship between compact contractible manifolds and their spines is complex.  Two non-homeomorphic compact contractible 4-manifolds can have homeomorphic spines, and a single compact contractible 4-manifold can have two non-homeomorphic spines - one of which can be divided into collapsible subpolyhedra while the other can't.  Thus, for compact contractible manifolds, a spine which can't be divided into collapsible subpolyhedra does not compel a non-splittable interior.  Specifically, a 4-ball whose interior is splittable has an disc spine, which can obviously be divided into two collapsible subpolyhedra that intersect in a collapsible subpolyhedron, and it has a dunce hat spine which, according to the Theorem of this paper, can't be divided in this fashion.  
Note that the preceding observation has the obvious consequence that the 4-ball and the Mazur 4-manifold which are not homeomorphic have spines which are homeomorphic - namely, dunce hats.


	We remark that in dimensions $n \geq 5,$ there exist splittable contractible open $n$-manifolds besides Euclidean $n$-space: the interior of every compact contractible $n$-manifold is splittable and every Davis $n$-manifold is splittable [AGS].  However, for $n \geq 5,$ no example of a non-splittable contractible open $n$-manifolds is known.

	The Theorem of this paper generalizes results of [GST] where it is shown that a specific triangulation of the \emph{original dunce hat} can't be expressed as the union of two collapsible subcomplexes.  (The original dunce hat arises when the map $f : \partial \Delta \to J$ is homotopic to a homeomorphism and the set $f^{-1}(\{v\})$ contains exactly 3 points.)
 
	During the preparation of this manuscript, the article [Bo] was posted.  The results of [Bo] significantly overlap the Theorem of this paper, but neither includes the other.  The techniques of [Bo] are quite different from techniques used here.  Furthermore, the movitation for the work in [Bo] (Lusternik-Schnirelmann category issues and the Andrews-Curtis conjecture) appears at first glance to be quite separate from the motivation for this paper (the splittability of the contractible open 4-manifolds), but the connection between these questions is provocative.

The authors thank Professor Craig Guilbault for helpful discussions and an idea that played a key role in the proof of Lemma 2.  

\section{The Theorem and constraints on its generalization}
   
\begin{definition}  Let $\Delta$ be a 2-dimensional disk, let $J$ be a simple closed curve and let $f : \partial \Delta \to J$  be a map.  Suppose there is a point $v \in J$ and an integer $p\geq 3$ such that $f^{-1}(\{v\})$ is a finite set with exactly $p$ elements, and $f$ maps each of the $p$ components of $\partial \Delta - f ^{-1}(\{v\})$ homeomorphically onto $J - \{v\}.$  (Hence, for each $y\in J,$ the set $f ^{-1}(\{y\})$ has exactly $p$ elements.)  The adjunction space $D = \Delta \cup_f J$ is called a \emph{generalized dunce hat}.  Identify $J$ with its image under the quotient map $\Delta \sqcup J\to D$ so that the quotient map restricted to $J$ is the identity.  (Here and subsequently, $\sqcup$ denotes disjoint union.)  Let $q :\Delta \to D$ denote the restriction of the quotient map $\Delta \sqcup J \to D$ to $\Delta.$  Observe that triangulations of $\Delta$ and $J$ can be chosen so that $\partial \Delta$ is a full subcomplex of $\Delta$ and $f : \partial\Delta\to J$ is simplicial.  Then these triangulations determine a triangulation of $D$ so that $J$ is a full subcomplex of $D$ and the quotient map $\Delta \sqcup J\to D$ is simplicial.  Call $v$ the \emph{singular vertex} of $D,$ and call $J$ the \emph{singular locus} of $D.$
\end{definition}
	
\begin{theorem}  No generalized dunce hat is the union of two proper subpolyhedra that each have finite first homology groups.
\end{theorem}

	Before proving this Theorem, we present a variety of evidence showing that the Theorem is, in some sense, a best possible result.  We show that if the hypotheses of the Theorem are weakened in two different ways, then a false statement results.  Also we show that the conclusion of the Theorem can't be strengthened to rule out expressing a generalized dunce hat as the union of \emph{three} proper subpolyhedra with trivial first homology.

	First, if the definition of \emph{generalized dunce hat} is broadened to allow attaching maps $f : \partial\Delta\to J$ for which the set $f ^{-1}(\{v\})$ has one or two elements, then it becomes possible in some cases to express the quotient space $\Delta \cup_f J$ as the union of two proper collapsible subpolyhedra.  For instance, if $f ^{-1}(\{v\})$ has one element, then $f : \partial\Delta\to J$ is a homeomorphism and $\Delta \cup_f J$ is a disk which can be expressed as the union of two proper subdisks.  If $f ^{-1}(\{v\})$ has two elements, then there are two possibilities: either $f : \partial\Delta\to J$ is null-homotopic or it is a degree two covering map.  In the case that $f : \partial\Delta\to J$ is null-homotopic, $\Delta \cup_f J$ is the quotient space obtained from the 2-sphere by identifying a pair of points.  In this situation, $\Delta \cup_f J$ can also be expressed as the union of two subdisks.  In the case that $f : \partial\Delta\to J$ is a degree two covering map, then $\Delta \cup_f J$ is a projective plane.  The Theorem holds for the projective plane; indeed, a Mayer-Vietoris sequence argument shows that the projective plane can't be expressed as the union of two proper subpolyhedra with finite first homology groups.

Second we consider the possibility of broadening the definition of generalized dunce hat by considering quotient spaces constructed by attaching a disk to a circle and allowing more than one singular vertex.  The following example shows that the Theorem becomes false if we allow two singular vertices.

\begin{example}  There is a quotient space $H = \Delta \cup_g J$ called the \emph{Jester's hat} that is similar to the dunce hat but which has two singular vertices instead of one, and which is the union of two proper collapsible subpolyhedra.  The Jester's hat originates in Pete Sparks 2014 Ph.D. thesis [S].  In more detail, let $w_0, w_1, \ldots , w_5$ be six distinct points of $\partial\Delta$ listed in cyclic order, and for $1 \leq i \leq 6,$ let $E_i$ denote the closure of the component of $\partial\Delta - \{ w_0, w_1, \ldots , w_5 \}$ with endpoints $w_{i-1}$ and $w_i$ (where $w_6 = w_0$).  Let $v_0$ and $v_1$ be two distinct points of $J,$ and let $J_1$ and $J_2$ be the closures of the two components of $J - \{ v_0, v_1 \}.$  Let $g : \partial\Delta\to J$ be a map such that $g^{-1}(\{v_0\}) = \{ w_0, w_2, w_4 \}$, $g^{-1}(\{v_1\}) = \{ w_1, w_3, w_5 \},$ $g$ maps each of $E_1,\ E_2,$ and $E_3$ homeomorphically onto $J_1,$ and $g$ maps each of $E_4,\ E_5,$ and $E_6$ homeomorphically onto $J_2.$  (This description of $g$ fixes $g$ up to isotopies that preserve $g|\{ w_0, w_1, \ldots , w_5 \}.$)  Then the quotient space $H = \Delta \cup_g J$ is called the \emph{Jester's hat}.  
Let $r : \Delta\to H$ denote the restriction of the quotient map $\Delta \sqcup J\to H$ to $\Delta.$  Let $L$ be an arc in $\Delta$ such that $\partial L = \{ w_0, w_3 \}$ and $\intr(L) \subset \intr(\Delta).$  Then $\Delta$ is the union of two disks $\Delta_1$ and $\Delta_2$ where $\partial\Delta_1 = E_1 \cup E_2 \cup E_3 \cup L,\  \partial\Delta_2 = E_4 \cup E_5 \cup E_6 \cup L$ and $\Delta_1 \cap \Delta_2 = L.$  Hence, $H = r(\Delta_1) \cup r(\Delta_2).$  Furthermore, $r(\Delta_1)$ and $r(\Delta_2)$ are easily seen to be collapsible.  (The map $r$ transports a collapse of $\Delta_1$ onto $E_1 \cup E_2 \cup E_3$ to a collapse of $r(\Delta_1)$ onto $J_1,$ and $J_1$ collapses to a point.  Similarly, $r(\Delta_2)$ collapses to $J_2$ which collapses to a point.)
\end{example}

Third, the following proposition shows that the conclusion of the Theorem can't be strengthened to rule out expressing a generalized dunce hat as the union of three proper subpolyhedra with trivial first homology.

\begin{prop}  Every generalized dunce hat is the union of three proper collapsible subpolyhedra.
\end{prop}

\begin{proof}  Let $D$ be a generalized dunce hat.  Let $J_1$ and $J_2$ be arcs in $J$ such that $\intr(J_1) \cup \intr(J_2) = J,$ and triangulate $D$ so that $J_1$ and $J_2$ are subcomplexes.  For $i = 1,\ 2,$ let $A_i$ be a simplicial neighborhood of $J_i$ in a second derived subdivision of this triangulation.  Then $A_i$ collapses to $J_i$ which in turn collapses to a point.  The set $D - (\intr(A_1) \cup \intr(A_2))$ is a compact subset of the open disk $q(\intr(\Delta)).$  Hence, this set is contained in the interior of a piecewise linear disk $A_3$ that lies in $q(\intr(\Delta)).$  $A_1,\ A_2$ and $A_3$ are collapsible subpolyhedra of $D$ that cover $D.$
\end{proof}
	Finally, the following example shows that the Theorem can't be strengthened to conclude that a generalized dunce hat $D$ can't be expressed as the union of two proper subpolyhedra $A$ and $B$ with the properties that the inclusion $A\to D$ induces the zero homomorphism on first homology and $B$ is collapsible.

\begin{example}  There is a generalized dunce hat $D = \Delta \cup_f J$ and proper subpolyhedra $A$ and $B$ of $D$ such that $A \cup B = D,$ the inclusion induced homomorphism $H_1(A)\to H_1(D)$ is zero and $B$ is collapsible.  Assign an orientation to $\partial\Delta$ and let $w_0,\ w_1,\ w_2,\ w_3$ be four distinct points of $\partial\Delta$ listed in orientation preserving cyclic order, and for $1 \leq i \leq 4,$ let $E_i$ denote the closure of the component of $\partial\Delta - \{ w_0, w_1, w_2 , w_3 \}$ with endpoints $w_{i-1}$ and $w_i$ (where $w_4 = w_0$).  Assign an orientation to  $J - \{v\}$.  Let $f : \partial\Delta\to J$ be a map such that $f ^{-1}(\{v\}) = \{ w_0, w_1, w_2, w_3 \},\ f|\intr(E_1) : \intr(E_1)\to J - \{v\}$ and $f|\intr(E_3) : \intr(E_3)\to J - \{v\}$ are orientation preserving homeomorphisms while $f|\intr(E_2) : \intr(E_2)\to J - \{v\}$ and $f|\intr(E_4) : \intr(E_4)\to J - \{v\}$ are orientation reversing homeomorphisms. Let $L$ be an arc in $\Delta$ such that $\partial L = \{ w_0, w_2 \}$ and $\intr(L) \subset \intr(\Delta).$ Let $B_1$ and $B_3$ be disjoint disks in $\Delta - L$ so that $K_1 = B_1 \cap \partial\Delta$ and $K_3 = B_3 \cap \partial\Delta$ are arcs in $\partial\Delta$ satisfying $w_1\in \intr(K_1),\ w_3\in \intr(K_3),\ K_1 \cap (E_3 \cup E_4) = \emptyset = K_3 \cap (E_1 \cup E_2)$ and $f(K_1 \cap E_1) = f(K_1 \cap E_2) = f(K_3 \cap E_3) = f(K_3 \cap E_4).$ Let $N = \cl(\Delta - (B_1 \cup B_3)).$  Then $N$ is a regular neighborhood of $L$ in $\Delta$ such that $N \cap \partial\Delta$ is the union of two disjoint arcs $K_0$ and $K_2$ such that $w_0 \in \intr(K_0),\ w_2 \in \intr(K_2),\ K_0 \cap (E_2 \cup E_3) = \emptyset = K_2 \cap (E_1 \cup E_4)$ and $f(K_0 \cap E_1) = f(K_0 \cap E_4) = f(K_2 \cap E_2) = f(K_2 \cap E_3).$ 
Let $A = q(N)$ and $B = q(B_1 \cup B_3).$ Then $A$ and $B$ are proper subpolyhedra of $D = A \cup B.$  $B$ is homeomorphic to a cone over a figure eight and, hence, is collapsible. There is a strong deformation retraction of $A$ onto $q(L).$ This deformation is constructed by beginning with a strong deformation retraction of the arc $f(K_0 \cap E_1)$  to its endpoint $\{v\}$, lifting this deformation to a strong deformation retraction of $N \cap \partial\Delta$ onto $\partial L$ via conjugation by $q^{-1},$ then extending the lifted deformation to a strong deformation retraction of $N$ onto $L,$ and finally pushing down the last deformation to a strong deformation retraction of $A$ onto $q(L)$ via conjugation by $q.$ Let $i : A\to D$ and $j : q(L)\to D$ denote inclusions.  Since there is a strong deformation retraction of $A$ onto $q(L),$ then there is a retraction $r : A\to q(L)$ such that $i$ is homotopic to $j\circ r$ in $D.$ We want to prove $i_* : H_1(A)\to H_1(D)$ is zero.  So it suffices to prove that $j_* : H_1(q(L))\to H_1(D)$ is zero. Let $E_\sharp = E_1 \cup E_2.$  Clearly, $q|E_\sharp = f|E_\sharp : (E_\sharp,\partial E_\sharp)\to (D,\{v\})$ is null homotopic rel $\partial E_\sharp.$ Hence, $(q|E_\sharp)_* : H_1(E_\sharp,\partial E_\sharp)\to H_1(D,\{v\})$ is zero. Let $k : L\to \Delta$ denote inclusion. Then $q\circ k = j\circ q|L.$ Clearly, $k : L\to \Delta$ is homotopic in $\Delta$ rel $\partial L$ to a map $\varphi : L\to \Delta$ such that $\varphi(L) = E_\sharp.$ Thus, $q\circ k : L\to D$ is homotopic in $D$ rel $\partial L$ to $(q|E_\sharp)\circ \varphi : L\to D.$ Thus, $(q\circ k)_* = ((q|E_\sharp)\circ \varphi)_* : H_1(L,\partial L)\to H_1(D,\{v\}).$ Since $(q|E_\sharp)_* = 0,$ it follows that $(j\circ q|L)_* = 0.$ Next, we will prove that $j_* : H_1(q(L),\{v\})\to H_1(D,\{v\})$ is zero by showing that $(q|L)_* : H_1(L,\partial L)\to H_1(q(L),\{v\})$ is an isomorphism. $q|L : (L,\partial L)\to (q(L),\{v\})$ factors as the composition of the quotient map $\pi : (L,\partial L)\to (L/\partial L,\partial L/\partial L)$ and a homeomorphism $\eta : (L/\partial L,\partial L/\partial L)\to (q(L),\{v\}).$ Proposition 2.22 on page 124 of [H] implies that $\pi$ induces an isomorphism on first homology. Clearly $\eta$ induces an isomorphism on first homology. Hence, $(q|L)_* : H_1(L,\partial L)\to H_1(q(L),\{v\})$ is an isomorphism and, therefore, $j_* : H_1(q(L),\{v\})\to H_1(D,\{v\})$ is zero.  The homology long exact sequences of the pairs $(q(L),\{v\})$ and $(D,\{v\})$ tell us that the inclusion induced homomorphisms $H_1(q(L))\to H_1(q(L),\{v\})$ and $H_1(D)\to H_1(D,\{v\})$ are isomorphisms.  Consequently, since the composition of $j_* : H_1(q(L))\to H_1(D)$ and $H_1(D)\to H_1(D,\{v\})$ equals the composition of $H_1(q(L))\to H_1(q(L),\{v\})$ and $j_* : H_1(q(L),\{v\})\to H_1(D,\{v\}),$ then it follows that $j_* : H_1(q(L))\to H_1(D)$ is zero. \hfill $\square$
\end{example}

\section{Auxiliary facts about maps between circles}  The results of this section are elementary facts about maps between circles for which we give elementary proofs.  Our goal is Proposition 5 which will be applied at a particular point in the proof of the Theorem.

\begin{definitions}  Let $f : K\to J$ be a map between circles.  Let $x \in K.$  If $f$ restricts to an embedding on a connected open neighborhood of $x,$ then we call $x$ a \emph{regular point} of $f.$  A point of $K$ which is not a regular point of $f$ is called a \emph{singular point} of $f.$  Let $R(f)$ denote the set of all regular points of $f$, and let $S(f)$ denote the set of all singular points of $f.$  Then $K = R(f) \sqcup S(f).$  
\end{definitions}

\begin{definitions} Let $f : K\to J$ be a map between oriented circles and let $x \in R(f).$  Then there is a connected open neighborhood $N$ of $x$ in $K$ such that $f|N : N\to J$ is an embedding. $f|N : N\to J$ must either preserve or reverse orientation.  If $f|N$ preserves orientation, then we say that $f$ is \emph{increasing at} $x.$ If $f|N$ reverses orientation, then we say that $f$ is \emph{decreasing at} $x.$  Let $R_+(f) = \{ x \in K : f \text{ is increasing at } x \}$ and let $R_-(f) = \{ x \in K : f \text{ is decreasing at }x \}.$  Then $R(f) = R_+(f) \sqcup R_-(f).$  Note that reversing the orientation on $J$ preserves the set $R(f)$ and interchanges the sets $R_+(f)$ and $R_-(f).$
\end{definitions}

If $f : K\to J$ is a map between oriented circles, observe that $R(f),\ R_+(f)$ and $R_-(f)$ are open subsets of $K$ and, therefore, $S(f)$ is a closed subset of $K.$  

\begin{definition}  Let $S \subset \mathbb{R}$ and let $f : S\to \mathbb{R}$ be a function. We say that $f$ is \emph{strictly increasing} if $f(y) < f(z)$ whenever $y, z \in S$ and $y < z.$  We say that f is \emph{locally strictly increasing} if every $x \in S$ has an open neighborhood $N_x$ in $\mathbb{R}$ such that $f|N_x \cap S : N_x \cap S\to \mathbb{R}$ is strictly increasing.  The notions of \emph{strictly decreasing} and \emph{locally strictly decreasing} are defined similarly.
\end{definition}

\begin{lemma} \label{locally strictly increasing} Let $S$ be a connected subset of $\mathbb{R}$ and let $f : S\to \mathbb{R}$ be a function.  If $f$ is locally strictly increasing, then it is strictly increasing.  Similarly, if $f$ is locally strictly decreasing, then it is strictly decreasing.  
\end{lemma}

\begin{proof}  Assume $f : S\to \mathbb{R}$ is locally strictly increasing.  Let $x, y \in S$ such that $x < y.$  Then $[x,y] \subset S$ and each $z \in S$ has an open neighborhood $N_z$ in $\mathbb{R}$ such that $f|N_z \cap S : N_z \cap S\to \mathbb{R}$ is strictly increasing. The open cover $N = \{ N_z : z \in S \}$ of $[x,y]$ has a Lebesgue number $\lambda.$  Choose $n \geq 0$ so that $x + n\lambda < y \leq x + (n+1)\lambda.$  Then for $1 \leq i \leq n,$ the two points $x + (i-1)\lambda$ and $x + i\lambda$ lie in an element of $N,$ and the two points $x + n\lambda$ and $y$ lie in an element of $N.$  Thus, $f(x) < f(x + \lambda) < f(x + 2\lambda) < \ldots < f(x + n\lambda) < f(y).$  We conclude that $f(x) < f(y).$  Thus, $f$ is strictly increasing.  
\\We leave the proof of the second assertion of Lemma \ref{locally strictly increasing} to the reader.
\end{proof}

\begin{proposition} \label{regular open intervals}  Suppose $f : K\to J$ is a map between oriented circles, $E$ is a connected open subset of $K$ and $E \subset R(f).$  Then either $E \subset R_+(f)$ or $E \subset R_-(f).$  Furthermore, if $f(E)$ is a proper subset of $J,$ then $f|E : E\to J$ is an embedding.
\end{proposition}

\begin{proof} Since the sets $E \cap R_+(f)$ and $E \cap R_-(f)$ are disjoint relatively open subsets of $E$ whose union is $E$ and $E$ is connected, then one of these sets must be empty.  The first conclusion of the proposition follows. 

Assume $f(E)$ is a proper subset of $J.$  Notice that since $f$ is regular at each point of $E,$ then $f|E : E\to J$ is an open map.  Hence, $f(E)$ is an open subset of $J.$  If $E = K,$ then $f(E)$ is a closed subset of $J.$  Since $J$ is connected, this would imply $f(E) = J$ which contradicts our hypothesis. We conclude that $E$ is a proper subset of $K.$ Hence, there is an orientation preserving homeomorphism $\varphi : (0,1)\to E$ and there is an orientation preserving embedding $\psi : (0,1)\to J$ such that $f(E) \subset \psi((0,1)).$  First assume $E \subset R_+(f).$ Then the composition $\psi^{-1}\circ f\circ \varphi : (0,1)\to (0,1)$ is locally strictly increasing. Lemma \ref{locally strictly increasing} now implies that $\psi^{-1}\circ f\circ \varphi$ is strictly increasing. Thus, $\psi^{-1}\circ f\circ \varphi|(0,1) : (0,1)\to (0,1)$ is injective.  Since $\varphi : (0,1)\to E$ is bijective, then $f|E : E\to J$ must be injective. Hence, $f|E : E\to J$ is an embedding. In the case that $E \subset R_-(f),$ reverse the orientation on $J.$  Now $E \subset R_+(f)$ and the preceding argument shows that $f|E : E\to J$ is an embedding. 
\end{proof}

\begin{definitions} Suppose $f : K\to J$ is a map between oriented circles, $x \in K$ and $x$ is an isolated point of $S(f).$  Then there is a connected open neighborhood $E$ of $x$ in K such that $E - \{x\} \subset R(f)$ and $E \neq K.$  Hence, there is an orientation preserving homeomorphism $\varphi : (-1,1)\to E$ such that $f(0) = x.$  Then $E_- = \varphi((-1,0))$ and $E_+ = \varphi((0,1))$ are the two components of $E - \{x\}.$  Proposition \ref{regular open intervals} implies that either $E_- \subset R_+(f)$ or $E_- \subset R_-(f),$ and either $E_+ \subset R_+(f)$ or $E_+ \subset R_-(f).$  If $E_- \subset R_+(f)$ and $E_+ \subset R_-(f),$ we say $f$ \emph{has a local maximum at} $x;$ and if $E_- \subset R_-(f)$ and $E_+ \subset R_+(f),$ we say f \emph{has a local minimum at} $x.$  The following proposition shows that the remaining two possibilities - either $E_- \cup E_+ \subset R_+(f)$ or $E_- \cup E_+ \subset R_-(f)$ - can't occur if $x \in S(f).$
\end{definitions}

\begin{proposition} \label{regular pts} If $f : K\to J$ is a map between oriented circles, $x \in K$ and $E$ is an open neighborhood of $x$ in $K$ such that either $E - \{x\} \subset R_+(f)$ or $E - \{x\} \subset R_-(f),$ then $x \in R(f).$
\end{proposition}

\begin{proof}  First assume $E - \{x\} \subset R_+(f).$  Let $\psi : (-1,1)\to J$ be an orientation preserving embedding such that $\psi(0) = f(x).$   There is an orientation preserving embedding $\varphi : (-1,1)\to K$ such that $\varphi(0) = x$ and $\varphi((-1,1)) \subset E.$  Thus, $\varphi((-1,1) - \{0\}) \subset E - \{x\}.$  Since $f$ is continuous, we can assume $f\circ \varphi((-1,1)) \subset \psi((-1,1)).$  Let $g = \psi^{-1}\circ f\circ \varphi : (-1,1)\to (-1,1).$   Since $E - \{x\} \subset R_+(f),$ then $g|(-1,0) : (-1,0)\to (-1,1)$ and $g|(0,1) : (-1,0)\to (-1,1)$ are orientation preserving maps.  Therefore, $g|(-1,0)$ and $g|(0,1)$ are locally strictly increasing and, hence, strictly increasing functions by Lemma \ref{locally strictly increasing}.  So $g|(-1,0)$ and $g|(0,1)$ are injective.  We assert that $g((-1,0)) \subset (-1,0).$  For suppose there is a $t \in (-1,0)$ such that $g(t) \geq 0.$  Choose $t' \in (t,0).$  Then $0 \leq g(t) < g(t') < 1.$  Hence, $g([t',0)) \subset [g(t'),1] \subset (0,1] .$  Since $g$ is continuous at $0,$ then $g(t'')$ approaches $g(0) = 0$ as $t'' \in [t',0)$ approaches 0 from below.  This implies $0 \in \cl(g[t',0)) \subset [g(t'),1],$ a contradiction.  A similar argument shows $g((0,1)) \subset (0,1).$  It follows that $g : (-1,1)\to (-1,1)$ is injective.  Since $g = \psi^{-1}\circ f\circ \varphi$ and $\varphi : (-1,1)\to K$ is injective, then $f|\varphi((-1,1))$ is injective.  Consequently, $x \in R(f).$  In the case that $E - \{x\} \subset R_-(f),$ reverse the orientation of $J.$  Then the preceding argument shows that $x \in R(f).$
\end{proof}

If $f : K\to J$ is a map between oriented circles and $x \in K$ has a connected open neighborhood $E$ in $K$ such that $E - \{x\} \subset R(f),$ then there are three possible cases: $E - \{x\} \subset R_+(f),$ $E - \{x\} \subset R_-(f),$ or one component of $E - \{x\}$ lies in $R_+(f)$ while the other component lies in $R_-(f).$  Proposition \ref{regular pts} implies that in either of the first two cases, $x \in R(f).$  Thus, if $x \in S(f),$ then the third case must hold.  Hence, we have:

\begin{corollary}  If $f : K\to J$ is a map between oriented circles, $x \in K$ and $x$ is an isolated point of $S(f),$ then $f$ has either a local maximum or a local minimum at $x.$ 
\end{corollary}
	We have reached the fact about maps between circles which is the goal of this section.

\begin{proposition} \label{maxs = mins} Suppose $f : K\to J$ is a map between oriented circles such that every singular point of $f$ is isolated.  Then $f$ has only finitely many singular points and the number of points at which $f$ has a local maximum is equal to the number of points at which $f$ has a local minimum.
\end{proposition}

\begin{proof} Since $K$ is compact and $S(f)$ is a closed subset of $K,$ then $S(f)$ is compact.  Since each point of $S(f)$ is isolated, it follows that $S(f)$ is a finite set.  Suppose $S(f) =  \{ w_1, w_2, \ldots , w_p \}$ indexed in orientation preserving cyclic order on $K.$  Let $w_0 = w_p$ and for $1 \leq i \leq p,$ let $E_i$ be the component of $K - S(f) = R(f)$ with endpoints $w_{i-1}$ and $w_i.$  Let $E_0 = E_p$ and $E_{p+1} = E_1.$  Observe that for $1 \leq i \leq p+1,$ Proposition \ref{regular open intervals} implies that either $E_i \subset R_+(f)$ or $E_i \subset R_-(f),$ and Proposition \ref{regular pts} implies that $E_{i-1}$ and $E_i$ can't both lie in either $R_+(f)$ or in $R_-(f).$  If follows that $p$ must be an even integer and we can assume that $E_1 \cup E_3 \cup \ldots \cup E_{p-1} \subset R_+(f)$ and $E_2 \cup E_4 \cup \ldots \cup E_p \subset R_-(f).$  Consequently, $f$ must have a local maximum at each point of the set $\{ w_1, w_3, \ldots , w_{p-1} \}$ and it must have a local minimum at each point of the set $\{ w_2, w_4, \ldots , w_p \}.$  Since these two sets have the same number of elements, the proposition is proved. 
\end{proof}	
%

\section{Auxiliary facts about components of submanifolds} \label{Aux facts} Here we list some facts that will be used without specific citation in the proof of the Theorem.

\begin{notation} We need notations that distinguish between the two types of \emph{interiors} that can be associated with a manifold that is the subset of a larger topological space.  If $A$ is a subset of a topological space $X,$ let $cl_X(A),$ $intr_X(A)$ and $fr_X(A)$ denote the \emph{closure}, \emph{interior} and \emph{frontier} of $A$ in $X,$ respectively.  Omit the subscript $X$ in contexts where no ambiguity can arise from this omission.  If $M$ is a topological manifold, then from this point on, let $in(M)$ and $\partial M$ denote the manifold interior and the manifold boundary of $M,$ respectively.
\end{notation}

	Recall that if $M \subset Q$ are manifolds of the same dimension, then \emph{Invariance of Domain} implies that $\inr(M) \subset \inr(Q).$  Hence, in this situation, $M \cap \partial Q = \partial M \cap \partial Q.$

\begin{definition}  If $M \subset Q$ are manifolds and $\dim(M) = n,$ we call $M$ a \emph{nice submanifold} of $Q$ if $\inr(M) \subset \inr(Q)$ and if $\partial M \cap \partial Q$ and $\partial M - \inr(\partial M \cap \partial Q)$ are ($n - 1$)-manifolds.
\end{definition}
	Suppose $T$ is a combinatorial triangulation of an $n$-manifold $M$ (i.e., the link of every vertex of $T$ is piecewise linearly homeomorphic to either a piecewise linear ($n - 1$)-sphere or a piecewise linear ($n - 1$)-ball.)  Let $K$ be a subcomplex of $T.$  Then regular neighborhood theory (
	[Br], Theorem 3.6, page 231) implies that the simplicial neighborhood of $K$ in a second derived subdivision $T''$ of $T$ is an $n$-dimensional nice submanifold of $M.$

\begin{notation}If $X$ is a topological space, let \emph{Comp}$(X)$ denote the set of all components of $X.$
\end{notation}
\begin{proposition} \label{inverses} Suppose $M \subset Q$ are manifolds such that $\inr(M) \subset \inr(Q).$  Then the functions $C \mapsto  \cl_M(C) : \comp(M \cap \inr(Q))\to \comp(M)$ and $D \mapsto D \cap \inr(Q) : \comp(M)\to \comp(M \cap \inr(Q))$ are inverse bijections.
\end{proposition}

\begin{proof}  It suffices to prove that these two functions are well defined and that each is a left inverse of the other.

First we prove that the function $C \mapsto \cl_M(C) : \comp(M \cap \inr(Q))\to \comp(M)$ is well defined, and that the function $D \mapsto D \cap \inr(Q)$ is a left inverse of $C \mapsto \cl_M(C).$  Let $C \in \comp(M \cap \inr(Q)).$  Then there is a $D \in \comp(M)$ such that $C \subset D.$  Therefore, $C \subset D \cap \inr(Q) \subset M \cap \inr(Q).$  $D$ is a relatively closed subset of $M$ and (because $M$ is locally connected) a relatively open subset of $M.$  Hence, $D$ is a manifold with the same dimension as $M.$  We draw two conclusions.  First, $\inr(D)$ is a dense subset of $D$ and, hence, $\cl_M(\inr(D)) = D.$  Second: $\inr(D) \subset \inr(M).$  The latter statement implies $\inr(D) \subset D \cap \inr(M) \subset D \cap \inr(Q).$  Since any two points of $D$ can be joined by an arc whose interior lies in $\inr(D),$ then it follows that $D \cap \inr(Q)$ is a connected set.  Now the inclusions $C \subset D \cap \inr(Q) \subset M \cap \inr(Q)$ together with the facts that $C$ is a component of M $\cap \inr(Q)$ and $D \cap \inr(Q)$ is connected imply that $C = D \cap \inr(Q).$  Since $\inr(D) \subset D$ and $\inr(D) \subset \inr(M) \subset \inr(Q),$ then $\inr(D) \subset C \subset D.$  Thus, $D = \cl_M(\inr(D)) \subset \cl_M(C) \subset \cl_M(D) = D.$  Hence, $\cl_M(C) = D \in \comp(M)$ and $\cl_M(C) \cap \inr(Q) = D \cap \inr(Q) = C.$

	Second we prove that the function $D \mapsto D \cap \inr(Q) : \comp(M)\to \comp(M \cap \inr(Q))$ is well defined, and that the function $C \mapsto \cl_M(C)$ is a left inverse of $D \mapsto D \cap \inr(Q).$  Let $D \in \comp(M).$  Then, as noted in the previous paragraph, $D$ is both a relatively closed and a relatively open subset of $M.$  Hence, $D \cap \inr(Q)$ is both a relatively closed and a relatively open subset of $M \cap \inr(Q).$  It was also noted in the previous paragraph that $D \cap \inr(Q)$ is connected.  It follows that $D \cap \inr(Q) \in \comp(M \cap \inr(Q)).$  The argument in the previous paragraph, shows that the closure in $M$ of an element of $\comp(M \cap \inr(Q))$ is the component of $M$ that contains it.  Since $D$ is the component of $M$ that contains $D \cap \inr(Q),$ then we conclude $\cl_M(D \cap \inr(Q)) = D.$ 
\end{proof}

\section{Proof of the Theorem} Assume that a generalized dunce hat $D$ is the union of two proper subpolyhedra $A$ and $B$ each of which has a finite first homology group.  To guide the reader through the upcoming argument, we briefly describe the eleven steps
that comprise it.
\begin{enumerate}[label=(\roman*)]
	\item Make $A$ connected by adding arcs.
	\item Prove $B$ does not contain $J$ and choose a point $b \in J - B.$
	\item Replace $A$ by a regular neighborhood, making $A \cap J$ a 1-dimensional submanifold of $J,$ $A - J$ a 2-dimensional submanifold of $D - J$ and $q^{-1}(A)$ a nice
2-dimensional submanifold of $\Delta.$
	\item Prove that each component of $q^{-1}(A)$ is a disk.
	\item Let $\mcal{E}$ denote the set of components of $A \cap J$ and let $\mcal{F}$ denote the set of components of $A - J.$  Construct a graph $S$ in $A$ with a vertex $w_E \in \inr(E)$ for each $E \in \mcal{E},$ a vertex $w_F \in \inr(F)$ for each $F \in \mcal{F},$ and an edge joining $w_E$ to $w_F$ if and only if $q^{-1}(w_E) \cap \cl(q^{-1}(F)) \neq \emptyset.$  ($S$ is essentially the nerve of the cover $\{ N(E) : E \in \mcal{E} \} \cup F$ where each $N(E)$ is a small regular neighborhood of $E$ in $A.$)  
\item Prove that $A$ strong deformation retracts onto $S,$ and conclude that $S$ is a tree.
\item Show that for each $F \in \mcal{F},$ $\cl(F)$ is a disk.
\item $A$ is the union of $\{ \cl(F) : F \in \mcal{F} \}$ and the intersection pattern of this collection of disks is encoded by the tree $S.$  Think of $A$ as a \emph{tree of disks}.
\item Prove Lemma \ref{mono on homology} which says that if a disk $G'$ in $D$ intersects $J$ in an arc $E'$ such that $v \notin G',$ $\partial E' \subset \partial G'$ and $b \in \inr(E') \subset \inr(G'),$ then the inclusion map $\partial G'\to D - \{b\}$ induces a monomorphism on first homology.
\item There is an $E \in \mcal{E}$ such that $b \in \inr(E).$  The remainder of the proof breaks into two cases: either there are two distinct elements $F_1$ and $F_2$ of $\mcal{F}$ such that $E \subset \cl(F_1)$ and $E \subset \cl(F_2),$ or there are not.  (The second case occurs only if $v \in \inr(E).$)
\item In the first case, we can find two subtrees $T_1$ and $T_2$ of $S$ such that $G_1 = \cup \{ \cl(F) : F \in \mcal{F}\text{ and }w_F\text{ is a vertex of }T_1 \}$ and $G_2 = \cup \{ \cl(F) : F \in \mcal{F}\text{ and }w_F$ is a vertex of $T_2 \}$ are disks with the property that $G_1 \cap G_2 = \partial G_1 \cap \partial G_2 = E$ and $(\partial G_1 - in(E)) \cup (\partial G_2 - \inr(E)) \subset \fr(A).$   Then $G = G_1 \cup G_2$ is a disk containing $b$ in its interior such that $\partial G \subset \fr(A) \subset B.$  It follows that the inclusion induced map $H_1(\partial G)\to H_1(D - \{b\})$ factors through a finite group.  The latter statement contradicts the conclusion of Lemma \ref{mono on homology} if we choose $G'$ to be a disk in $\inr(G)$ satisfying the hypotheses of Lemma \ref{mono on homology}.
\item In the second case, we can find a single subtree $T_1$ of $S$ such that $G_1 = \cup \{ \cl(F) : F \in \mcal{F}\text{ and }w_F\text{ is a vertex of }T_1 \}$ is a disk containing $b$ in its interior with the property that $\partial G_1 \subset \fr(A) \subset B.$  As in the previous case, it follows that the inclusion induced map $H_1(\partial  G)\to H_1(D - \{b\})$ factors through a finite group.  Again the latter statement contradicts the conclusion of Lemma \ref{mono on homology} if we choose a disk $G'$ in $\inr(G)$ which satisfies the hypotheses of Lemma \ref{mono on homology}.
\end{enumerate}
	If you understand these eleven steps completely, read no further.  For everyone else, we provide a detailed proof.

We do not initially suppose that $A$ is connected.  We can make $A$ connected by joining distinct components with arcs (one fewer arcs than the number of components).  According to the reduced Mayer-Vietoris sequence, this process doesn't alter $A$'s first homology.  So we can assume $A$ is connected. 

Neither $A$ nor $B$ can contain $J.$  For suppose $B$ contains $J.$  Choose a point $x \in D - B.$  Then $x \in D - J.$  There is a point $y \in \inr(\Delta)$ such that $q ^{-1}(\{x\}) = \{y\}.$  Any retraction of $\Delta - \{y\}$ onto $\partial \Delta$ induces a retraction of $D - \{x\}$ onto $J.$  This retraction restricts to a retraction $r$ of $B$ onto $J.$  Then $r\circ i = \id_J$ where $i : J\to B$ denotes the inclusion.    Hence, the identity map $H_1(J)\to H_1(J)$ factors through the finite group $H_1(B).$  Since $H_1(J) \approx Z,$ this is impossible.  Thus, $B$ can't contain $J.$  $A$ can't contain $J$ for similar reasons.  

Because $B$ doesn't contain $J,$ then we can choose a point $b \in J - B$ such that $b \neq v.$  Since $A \cup B = D,$ then necessarily $b \in A.$

Let $T$ be a triangulation of $D$ so that subcomplexes of $T$ triangulate $\{v\},\ A,\ B$ and $\{b\}$ and so that there is a triangulation $q^{-1}(T)$ of $\Delta$ that makes $q : \Delta\to D$ a simplicial map from $q^{-1}(T)$ to $T.$  Let $T''$ be a second derived subdivision of $T.$  Then there is a second derived subdivision $q^{-1}(T'')$ of $q^{-1}(T)$ so that $q : \Delta\to D$ is simplicial from $q^{-1}(T'')$ to $T''.$  We enlarge $A$ by replacing it by its simplicial neighborhood with respect to $T''.$   This process replaces $q ^{-1}(A)$ by its simplicial neighborhood with respect to $q^{-1}(T'').$  Then according to the remark following the definition of \emph{nice submanifold} in section \ref{Aux facts}, $A \cap J$ is now a 1-dimensional submanifold of $J$ and $q^{-1}(A)$ is now a nice 2-dimensional submanifold of $\Delta.$ Thus, $q^{-1}(A) \cap \partial \Delta$ is a 1-dimensional submanifold of $\partial \Delta,$ and $q^{-1}(A) \cap \inr(\Delta)$ is a 2-dimensional submanifold of $\inr(\Delta).$  Since $q$ maps $\inr(\Delta)$ homeomorphically onto $D - J,$ then it follows that $A - J$ a 2-dimension submanifold of $D - J.$ 

We devote the next three paragraphs to describing the topology of $A$ at points of $A \cap J.$  Let $\mcal{E}$ denote the set of components of $A \cap J.$  Since $A \cap J$ is a 1-dimensional submanifold of $J,$ then each element of $\mcal{E}$ is an arc in $J.$  Let $E \in \mcal{E}.$  Since $v$ is a vertex of the triangulation $T$ of $D$ and $E$ arises by taking a simplicial neighborhood in a second derived subdivision $T''$ of $T,$ then either $v \in \inr(E)$ or $v \notin E.$  Similarly, either $b \in \inr(E)$ or $b \notin E.$  Recall that associated to the map $f = q|\partial \Delta : \partial \Delta\to J$ is an integer $p \geq 3$ such that $q^{-1}(\{v\})$ has exactly $p$ elements, and $f$ maps each of the $p$ components of $\partial \Delta - q^{-1}(\{v\})$ homeomorphically onto $J - \{v\}.$  It follows that $q^{-1}(E)$ is the union of $p$ disjoint arcs $E_1,\ \ldots\ ,\ E_p$ in $\partial \Delta.$  

First focus on the case in which $v \notin E.$  In this situation $q$ maps each $E_i$ homeomorphically onto $E.$  Let $N$ be a regular neighborhood of $E$ in $A.$  Then $q^{-1}(N)$ is a regular neighborhood of $q^{-1}(E)$ in $q^{-1}(A).$  Since $q^{-1}(A)$ is a nice 2-dimensional submanifold of $\Delta,$ it follows that $q^{-1}(N)$ is the union of $p$ disjoint disks $N_1,\ \ldots\ ,\ N_p$ with $N_i \cap \partial \Delta = E_i$ for $1 \leq i \leq p.$  Furthermore, $q$ embeds each $N_i$ in $N$ such that $q(N_i) \cap q(N_j) = E$ for $i \neq j.$  Thus, $N$ is an \emph{open book} with \emph{pages} $q(N_1),\ \ldots\ ,\ q(N_p)$ and binding $E.$  

Next consider the case in which $v \in \inr(E).$  Then for $1 \leq i \leq p,$ there is a point $v_i \in \inr(E_i)$ such that $q^{-1}(\{v\}) \cap E_i  = \{v_i\}.$  The set of singular points of the map $q|\partial \Delta : \partial \Delta\to J$ lies in the $p$-element set $q^{-1}(\{v\}).$  If we orient $\partial \Delta$ and $J,$ then Proposition \ref{maxs = mins} implies that the number of points at which the $q|\partial \Delta : \partial \Delta\to J$ has a local maximum equals the number of points at which it has a local minimum.  Hence, if we let $E'$ and $E''$ be closures of the two components of $E - \{v\},$ then there is an integer $r \geq 0$ such that $2r \leq p$ 
and (after permuting subscripts):
\begin{enumerate}[label=(\roman*)]
	\item $q$ maps the closure of each component of $E_i - \{v_i\}$ onto $E'$ for $1 \leq i \leq r,$
	\item $q$ maps the closure of each component of $E_i - \{v_i\}$ onto $E''$ for $r+1 \leq i \leq 2r,$ and
	\item $q$ maps $E_i$ homeomorphically onto $E$ for $2r+1 \leq i \leq p.$
\end{enumerate}
(Note that if $2r = p$ so that $f$ has either a local maximum or a local minimum at each point of $q^{-1}(\{v\}),$ then $f$ can be perturbed by homotopy on a small neighborhood of $q^{-1}(\{v\})$ so that its image misses $v$ making $f : \partial \Delta\to J$ homotopic to a constant map.)  Continuing our analysis in the case that $v \in \inr(E),$ again let $N$ be a regular neighborhood of $E$ in $A.$  Then $q^{-1}(N)$ is a regular neighborhood of $q^{-1}(E)$ in $q^{-1}(A).$  Again since $q^{-1}(A)$ is a nice 2-dimensional submanifold of $\Delta,$ it follows that $q^{-1}(N)$ is the union of $p$ disjoint disks $N_1,\ \ldots\ ,\ N_p$ with $N_i \cap \partial \Delta = E_i$ for $1 \leq i \leq p.$  Furthermore, $q$ embeds each $N_i - E_i$ in $N - E$ such that $q(N_i - E_i) \cap q(N_j - E_j) = \emptyset$ for $i \neq j.$  Thus, $N$ is an \emph{open pop-up book} in which the pages $q(N_1),\ \ldots\ ,\ q(N_r)$ are cones attached along one half of the binding $E,$ the pages $q(N_{r+1}),\ \ldots\ ,\ q(N_{2r})$ are cones attached along the other half of the binding $E,$ and the pages $q(N_{2r+1}),\ \ldots\ ,\ q(N_p)$ are disks attached along the entire binding $E.$  


Next we argue that each component of $q^{-1}(A)$ is a disk.  We first observe that every component of $q^{-1}(A)$ must intersect $\partial \Delta.$  For suppose $C$ is a component of $q^{-1}(A)$ that is disjoint from $\partial \Delta.$  Then $C$ and $q^{-1}(A) - C$ are disjoint closed subsets of $\Delta$ that are mapped by $q$ to disjoint closed subsets of $D$ whose union is $A.$  This contradicts our earlier supposition that $A$ is connected.

Now suppose $C$ is a component of $q^{-1}(A)$ that is not a disk.  Then $C$ is a disk with holes and there is a component $K$ of $\partial C$ that lies $\inr(\Delta)$ and bounds a disk $F$ in $\inr(\Delta).$  $\inr(F) \cap q^{-1}(A) = \emptyset$ because every component of $q^{-1}(A)$ intersects $\partial \Delta.$   Since $\Delta - \inr(F)$ is an annulus with boundary components $\partial \Delta$ and $K,$ and since $q^{-1}(A) \cap \partial \Delta$ is a proper subset of $\partial \Delta,$ then there is a retraction of $\Delta - \inr(F)$ onto $K$ that maps $q^{-1}(A) \cap \partial \Delta$ to a single point $z \in K.$  We restrict this retraction to $C$ to obtain a retraction of $C$ onto $K$ that maps $C \cap \partial \Delta$ to $z.$  We conjugate this retraction by $q^{-1}$ to obtain a retraction of $q(C)$ onto $q(K)$ that maps $q(C) \cap J$ to $q(z).$  This retraction extends to a retraction $r$ of $A$ onto $q(K)$ which sends $A - q(C \cap \inr(\Delta))$ to $q(z).$  Then $r\circ j = \id_q(K)$ where $j : q(K)\to A$ denotes the inclusion.  Hence, the identity map $H_1(q(K))\to H_1(q(K))$ factors through the finite group $H_1(A).$  Since $H_1(q(K)) \approx Z,$ this is impossible.  Thus, each component of $q^{-1}(A)$ must be a disk.  

The next step of the proof is to construct a graph $S$ in $A$ which is a spine of $A$ in the sense that there is a strong deformation retraction of $A$ onto $S.$  Recall that $\mcal{E}$ denotes the set of components of $A \cap J,$ and each element of $\mcal{E}$ is an arc in $J.$  Let $\mcal{F}$ denote the set of components of $A - J.$  Then each $F \in \mcal{F}$ is a 2-dimensional submanifold of $D - J,$ and $\cl(F)$ is a compact subpolyhedron of $D$ such that each component of $\cl(F) \cap J$ is either an element of $\mcal{E}$ or a subarc of an element $E$ of $\mcal{E}$ joining one endpoint of $E$ to the point $v$ which lies in $\inr(E).$

 $S$ has two types of vertices: 
\begin{enumerate}[label=(\roman*)]	
\item a point $w_E \in \inr(E)$ for each $E \in \mcal{E}$ such that $w_E = v$ if $v \in \inr(E),$ and 
\item a point $w_F \in \inr(F)$ for each $F \in \mcal{F}.$ 
\end{enumerate}
 
In other words, $\{ w_E : E \in E \} \cup \{ w_F : F \in \mcal{F} \}$ is the set of vertices of $S.$  For each $F \in \mcal{F},\ q^{-1}(F)$ is a component of $q^{-1}(A) \cap \inr(\Delta)$ because $q$ maps $q^{-1}(A) \cap \inr(\Delta)$ homeomorphically onto $A - J.$  Let $P(F) = \cl(q^{-1}(F)).$  Then Proposition \ref{inverses} implies that $P(F)$ is a component of $q^{-1}(A).$  Hence, $P(F)$ is a nice 2-dimensional submanifold of $\Delta$ and, in particular, as we argued in a previous paragraph, $P(F)$ is a disk.  Furthermore, $q^{-1}(w_F) \in \inr(P(F)).$  Let $B(F) = P(F) \cap q^{-1}( \{ w_E : E \in \mcal{E} \} ).$  Then $B(F)$ is a finite subset of $P(F) \cap \partial \Delta \subset \partial P(F).$  Let $K(F)$ be a cone in $P(F)$ with vertex $q^{-1}(w_F)$ and base $B(F)$ such that $K(F) \cap \partial P(F) = B(F).$  Then $q(K(F))$ is a graph in $\cl(F).$  Let $S = \cup \{ q(K(F)) : F \in \mcal{F} \}.$  Then $S$ is a graph in $A.$

There is a strong deformation retraction of $A$ onto $S$ which is constructed in several steps.  Begin with a strong deformation retraction of $A \cap J$ onto $\{ w_E : E \in \mcal{E} \}.$  Lift this strong deformation retraction (via conjugation by $q^{-1}$) to a strong deformation retraction of $q^{-1}(A \cap J)$ onto $q^{-1}(\{ w_E : E \in \mcal{E} \} ).$  For each $F \in \mcal{F},$ this strong deformation retraction restricts to a strong deformation retraction of $P(F) \cap \partial \Delta$ onto $B(F)$ which, in turn, extends to a strong deformation retraction of $P(F)$ onto $K(F).$  This strong deformation retraction pushes down (via conjugation by $q$) to a strong deformation retraction of $\cl(F)$ onto $q(K(F))$ which extends the original strong deformation retraction of $\cl(F) \cap J$ onto $(\{ w_E : E \in \mcal{E} \}) \cap \cl(F).$  The union of these strong deformation retractions is a strong deformation retraction of $A$ onto $S.$

The end of the strong deformation retraction of $A$ onto $S$ is a retraction map $r : A\to S.$  Then $r\circ \sigma = \id_S$ where $\sigma : S\to A$ denotes the inclusion.  Hence, the identity map $H_1(S)\to H_1(S)$ factors through the finite group $H_1(A).$  Thus, $H_1(S)$ must be a finite group.  It follows that since $S$ is a graph, then $H_1(S) = 0$ and $S$ must be a tree.
	
	Our next goal is to show that $\cl(F)$ is a disk for each $F \in \mcal{F}.$  Let $F \in \mcal{F}.$  $q$ maps the disk $P(F)$ onto $\cl(F),$ and $q$ maps $P(F) - \partial \Delta$ homeomorphically onto $\cl(F) - J.$  It remains to analyze the behavior of $q$ on the components of $P(F) \cap \partial \Delta.$  Since $S$ is a tree, then $q$ can't identify two distinct points of $B(F).$  So if $x \in q^{-1}(\{w_E\})$ and $x' \in q^{-1}(\{w_E'\})$ are distinct points of $B(F)$ where $E$ and $E' \in \mcal{E},$ then $w_E \neq w_E'$ and, hence, $E \neq E'.$  Therefore, $q$ must map distinct components of $P(F) \cap \partial \Delta$ into distinct components of 
$A \cap J.$  (Said another way: if $E \in \mcal{E}$ and $\cl(F) \cap E \neq \emptyset,$ then $P(F)$ contains exactly one of  
the $p$ components of $q^{-1}(E)$ and is disjoint from all the others.)  Suppose $E \in \mcal{E}$ and $\cl(F)$ 
$\cap E \neq \emptyset.$  There are two possible cases.\\ 
\begin{enumerate}[label=(\roman*)]
	\item $E \subset \cl(F).$  In this case, $P(F)$ contains a unique component $E^*$ of $q^{-1}(E)$ and $q$
maps $E^*$ homeomorphically onto $E.$ \\
	\item $E \not\subset \cl(F).$  (This case occurs only if $v \in \inr(E).$  Hence, there is at most one $E \in \mcal{E}$ for which $\cl(F) \cap E \neq \emptyset$ and $E \not\subset \cl(F).$)  In this case, $E' = E \cap \cl(F)$ is the closure of one of the two components of $E - \{v\}.$  Again $P(F)$ contains a unique component $E^*$ of $q^{-1}(E).$  Furthermore, there is a unique $v^* \in \inr(E^*)$ such that $q(v^*) = v,$ and $q$ maps each component of $E^* - \{v^*\}$ homeomorphically onto $E' - \{v\}.$   In this situation, we might say that $q$ \emph{zips up} $E^*,$ mapping the open arc $\inr(E^*)$ 
that lies in $\partial P(F)$ to the half open arc $\inr(E') \cup \{v\}$ that lies in $\inr(\cl(F)$).
\end{enumerate}

$q$ maps $P(F)$ onto $\cl(F)$ and $q$ maps $P(F) - \partial \Delta$ homeomorphically onto $\cl(F) - J.$  In the event that case ii) doesn't occur, then $q$ maps each component of $P(F) \cap \partial \Delta$ homeomorphically onto a corresponding component of $\cl(F) \cap J.$  In this situation $q$ maps $P(F)$ homeomorphically onto $\cl(F),$ making $\cl(F)$ a disk.  In the event that case ii) occurs (for only one $E \in \mcal{E}$ for which $\cl(F) \cap E \neq \emptyset$), then $q$ maps all but one of the components of $P(F) \cap \partial \Delta$ homeomorphically onto a corresponding component of $\cl(F) \cap J,$ and $q$ \emph{zips up} the remaining component of $P(F) \cap \partial \Delta$ as described above.  In this situation, although $q|P(F) : P(F)\to \cl(F)$ is not a homeomorphism, nonetheless $q|P(F)$ does not alter the topological type of $P(F).$  Hence, $\cl(F)$ is also a disk in the event that case ii) occurs.

Since $A$ is the union of the collection of disks $\{ \cl(F) : F \in \mcal{F} \}$ and the intersection pattern of $\{ \cl(F) : F \in \mcal{F} \}$ is encoded by the tree $S,$ we call $A$ a \emph{tree of disks}.

Recall that the point $b$ lies in $J - B$ and $b \neq v.$  Hence, $b \in A$ and there is an $E \in \mcal{E}$ such that $b \in \inr(E).$  

To complete the proof of the Theorem, we need:

\begin{lemma}\label{mono on homology}If $G'$ is a disk in $D$ which intersects $J$ in an arc $E'$ such that $v \notin G',$ $\partial E' \subset \partial G'$ and $b \in \inr(E') \subset \inr(G'),$ then the inclusion map $j :  \partial G'\to D - \{b\}$ induces a monomorphism on first homology.
\end{lemma}

\begin{proof}  Let $G_1'$ and $G_2'$ be the closures of the two components of $G'- E'.$  Then $G_1'$ and $G_2'$ are disks in $D$ such that $G_i \cap J = E' \subset \partial G_i'$ for $i = 1, 2$ and $G_1' \cap G_2' = E'.$  For $i = 1, 2,$ let $G_i^* = \cl(q^{-1}(G_i' - E')).$  Then $G_i^*$ is a disk in $\Delta$ which $q$ maps homeomorphically onto $G_i'.$  Also $E_i^* = G_i^* \cap \partial \Delta$ is an arc in $\partial \Delta$ which is a component of $q^{-1}(E').$  Since different components of $q^{-1}(E')$ must lie in distinct components of $\partial \Delta - q^{-1}(\{v\}),$ then $E_1^*$ and $E_2^*$ lie in distinct components of $\partial \Delta - q^{-1}(\{v\}).$

Form a cone $T$ in $\Delta$ over the $p$-point set $q^{-1}(\{v\})$ with vertex at a point $z \in \inr(\Delta)$ such that $T \cap \partial \Delta = q^{-1}(\{v\}).$  ($T$ is the union of $p$ arcs joining $z$ to the points of $q^{-1}(v)$ such that the interiors of the arcs lie in $\inr(\Delta)$ and are disjoint from each other.)  There is a homeomorphism of $\Delta$ which fixes $\partial \Delta$ and pulls $G_1^*$ and $G_2^*$ off $T.$  The inverse of this homeomorphism moves $T$ off $G_1^* \cup G_2^*.$  So we can assume $T$ is disjoint from $G_1^* \cup G_2^*.$  Observe that $q(T)$ is the suspension of $p$ points with \emph{poles} at $q(z)$ and $v.$

Since $E_1^*$ and $E_2^*$ lie in distinct components of $\Delta - T,$ then so do $G_1^*$ and $G_2^*.$  Let $L_1,\ L_1',\ L_2$ and $L_2'$ be arcs of $T$ (joining $z$ to points of $q^{-1}(\{v\})$) labeled so that $L_i \cup L_i'$ is the frontier of the component of $\Delta - T$ that contains $G_i^*$ for $i = 1, 2.$  Note that because $p \geq 3,$ the sets $\{ L_1, L_1' \}$ and $\{ L_2, L_2' \}$ have at most one element in common.  So we can assume that $L_1,\ L_1',\ L_2$ and $L_2'$ are labeled so that $L_1 \notin \{ L_2, L_2' \}$ and $L_2 \notin \{ L_1, L_1' \}.$  


There is a unique retraction of $\partial \Delta - q^{-1}(\{b\})$ onto $q^{-1}(\{v\})$ which extends to a retraction $r^*$ of $\Delta - q^{-1}(\{b\})$ onto $T$ that maps $\partial G_i^* - \inr(E_i^*)$ homeomorphically onto $L_i \cup L_i'$ for $i = 1, 2.$  $r^*$ induces a retraction $r$ of $D - \{b\}$ onto $q(T)$ (via conjugation by $q$) that maps $J - \{b\}$ to $\{v\}$ and maps the open arc $\partial G_i' - E'$ homeomorphically onto the open arc $q(L_i \cup L_i') - \{v\}$ for $i = 1, 2.$

  	Let $C$ denote the wedge of two circles $C_1$ and $C_2$ which intersect in the one-point set $\{c\}.$  Let $s : q(T)\to C$ be a map which maps $q(\inr(L_i))$ homeomorphically onto $C_i - \{c\}$ for $i = 1, 2$ and which maps $q(T - (\inr(L_1) \cup \inr(L_2)))$ to $\{c\}.$  $H_1(C) \approx Z \oplus Z$ and there is a generating set $\{\gamma_1, \gamma_2\}$ for $H_1(C)$ such that the subgroup of $H_1(C)$ generated by $\gamma_i$ is the image of the inclusion induced map $H_1(C_i)\to H_1(C).$  Observe that $s\circ r\circ j|\partial G_i' - \inr(E')$ is homotopic rel endpoints to a map which takes $\partial G_i' - E'$ homeomorphically onto $C_i - \{c\}$ and takes $\partial E'$ to $\{c\}.$  Thus, the induced map $s_*\circ r_*\circ j_* : H_1(\partial G')\to H_1(C)$ maps a generator of $H_1(\partial G') \approx Z$ onto one of $\pm(\gamma_1 + \gamma_2)$ or $\pm(\gamma_1 - \gamma_2).$  Since these elements of $H_1(C)$ are of infinite order, then $s_*\circ r_*\circ j_*$ is a monomorphism.  It follows that $j_* : H_1(\partial G')\to H_1(D - \{x\})$ must be a monomorphism. 
\end{proof}
Finally we come to the heart of the argument.  Recall $b \in J - B \subset A,$ $b \neq v$ and there is an $E \in \mcal{E}$ such that $b \in \inr(E).$  The vertex $w_E$ of the tree $S$ also lies in $\inr(E).$  $w_E$ is the endpoint of $p$ different edges of $S$ whose other endpoints are of the form $w_{F_1},\ \ldots\ , w_{F_p}$ where $F_1,\ \ldots ,\ F_p$ are distinct elements of $\mcal{F}$ such that $E \cap \cl(F_i) \neq \emptyset.$  Hence, $S - \{w_E\}$ has $p$ components whose closures we label $S_1,\ \ldots ,\ S_p$ so that $w_{F_i}$ is a vertex of $S_i$ for $1 \leq i \leq p.$  Each $S_i$ is a tree that has $w_E$ as an endpoint.  For $1 \leq i \leq p,$ let $A_i = \cup \{ \cl(F) : F \in \mcal{F}$ and $w_F$ is a vertex of $S_i \}.$  Then $A = A_1 \cup \ldots \cup A_p$ and $A_i \cap A_j \subset E$ for $i \neq j.$  

  We break the remainder of the proof into two cases. 
\begin{enumerate}[label=(\roman*)]
\item Either $v \notin E,$ or $v \in \inr(E)$ and there are at least two distinct elements $F$ of $\mcal{F}$ such 
that $E \subset \cl(F).$  
\item $v \in \inr(E)$ but there is at most one element $F$ of $\mcal{F}$ such that $E \subset \cl(F).$
\end{enumerate}

\textbf{Case i): $v \notin E,$ or $v \in \inr(E)$ and there are at least two distinct elements $F$ of $\mcal{F}$ such that $E \subset \cl(F).$}  In the situation that $v \notin E, E \subset \cl(F_i)$ for $1 \leq i \leq p.$  Also since $(A_i - E) \cap (A_j - E) =  \emptyset$ for $i \neq j,$ then $v$ belongs to at most one of the sets $A_i - E.$  Therefore, since $p \geq 3,$ we can assume that, after reindexing, $v \notin A_1 \cup A_2.$  Consequently, in Case i), we can assume $E \subset \cl(F_1),\ E \subset \cl(F_2),$ and either $v \notin A_1 \cup A_2$ or $v \in \inr(E).$

 Now for $i = 1, 2,$ let $T_i$ be a maximal subtree of $S_i$ that contains $w_E$ and has order 2 at every other vertex of $T_i$ of the form $w_E'$ where $E' \in \mcal{E}.$  In other words, to construct $T_i,$ orient the edges of $S_i$ away from $w_E,$ include the edge from $w_E$ to $w_{F_i}$ and all the edges emanating from $w_{F_i},$ but only one edge emanating from a vertex $w_{E'}$ that is at the terminal end of an edge emanating from $w_{F_i}.$  Continue in this fashion to construct $T_i$ inductively: when one encounters a new vertex at the terminal end of an edge, add only one outgoing edge if the vertex is of the form $w_{E'}$ ($E'\in \mcal{E}$), but add all the outgoing edges if the vertex is of the form $w_{F'}$ ($F' \in \mcal{F}$).   

For $i = 1, 2,$ let $G_i = \cup \{ \cl(F) : F \in \mcal{F}$ and $w_F$ is a vertex of $T_i \}.$  Then $F_i \subset G_i \subset A_i.$  Moreover, each $G_i$ is a union of finitely many disks such that the intersection of any two of the disks is either empty or an arc lying in the boundary of each, and the intersection of any three distinct disks is empty.  Thus, each $G_i$ is a compact 2-manifold.  The strong deformation retraction of $A$ onto $S$ restricts to a strong deformation retraction of $A_i$ onto $S_i$ which, in turn, restricts to a strong deformation retraction of $G_i$ onto $T_i.$  Thus, each $G_i$ is contractible and, hence, is a disk.  Clearly, $E$ is an arc in $\partial G_i.$  


Observe that $\fr(A) = \cup \{ \cl_D(\partial F) : F \in \mcal{F} \}.$  Also observe that for $i = 1, 2,$ $\partial G_i = ( \cup \{ \cl_D(\partial F) : F \in \mcal{F}$ and $w_F$ is a vertex of $T_i \} ) \cup \inr(E).$  Thus, $\partial G_i - \inr(E) \subset \fr(A).$

Let $G = G_1 \cup G_2.$  Since $G_1 \cap G_2 = E,$ then $G$ is a disk, $\partial G = (\partial G_1 - \inr(E)) \cup (\partial G_2 - \inr(E))$ and $\inr(E) \subset \inr(G).$  Hence, $\partial G \subset \fr(A).$  Since $D = A \cup B$ and $B$ is a closed subset of $D,$ then $\fr(A) \subset B.$  Therefore, $\partial G \subset B.$  Furthermore, either $v \notin G$ if  $v \notin E,$ or $v \in \inr(G)$ if $v \in \inr(E).$

Since $b \in \inr(E) - \{v\} \subset \inr(G),$ we can choose an arc $E'$ in $\inr(E) - \{v\}$ such that $b \in \inr(E'),$ and for $i = 1, 2,$ we can choose a disk $G_i' \subset \cl(F_i)$ such that $G_i' \cap \partial (\cl(F_i)) = E'.$  Hence, $E' \subset \partial G_1' \cap \partial G_2'$ and $G_1' \cap G_2' = E'.$  Thus, $G_i'- E' \subset F_i.$  Therefore, $G'= G_1' \cup G_2'$ is a disk in $\inr(G)$ such that $v \notin G'$ and $b \in \inr(E') \subset \inr(G').$  Let $j : \partial G'\to D - \{b\}$ denote the inclusion map.  Since $G - \inr(G')$ is an annulus in $D - \{b\}$ with boundary components $\partial G$ and $\partial G',$ then $j$ is homotopic in $D - \{b\}$ to a map $k : \partial G'\to D - \{b\}$ such that $k(\partial G') \subset \partial G.$  Hence, the induced maps $j_* : H_1(\partial G')\to H_1(D - \{b\})$ and $k_* : H_1(\partial G')\to H_1(D - \{b\})$ are equal.  Since $\partial G \subset B \subset D - \{b\},$ then $k_*$ factors through the finite group $H_1(B).$  Since $H_1(\partial G') \approx Z,$ it follows that $k_*$ and, hence, $j_*$ are not monomorphisms, contradicting Lemma \ref{mono on homology}.  This concludes the proof of the Theorem in Case i). \qed 

\textbf{Case ii): $v \in \inr(E)$ but there is at most one element $F$ of $\mcal{F}$ such that $E \subset \cl(F).$}  Let $E_1$ and $E_2$ denote the closures of the two components of $E - \{v\}.$  Since $b \in \inr(E) - \{v\},$ then we can assume $b \in \inr(E_1).$  There are $p$ distinct elements $F$ of $\mcal{F}$ such that $E \cap \cl(F) \neq \emptyset$ where $p \geq 3.$  Hence, in this case there is an $F \in \mcal{F}$ such that $E \cap \cl(F) \neq \emptyset$ but $E \not\subset \cl(F).$  In this situation, $E \cap \cl(F)$ equals either $E_1$ or $E_2.$  If we orient $\partial \Delta$ and $J,$ then one of the two outcomes $E \cap \cl(F) = E_1$ and $E \cap \cl(F) = E_2$ represents a local maximum of the map $q|\partial \Delta : \partial \Delta\to J$ at a point of $\partial P(F) \cap q^{-1}(E)$ while the other represents a local minimum.  According to Proposition \ref{maxs = mins}, the number of points at which $q|\partial \Delta$ has local maxima equals the number of points at which it has local minima.  Consequently, since at least one of the two outcomes - $E \cap \cl(F) = E_1$ and $E \cap \cl(F) = E_2$ - occurs, then both must occur.  Therefore, we can choose $F_1 \in \mcal{F}$ such that $E \cap \cl(F_1) = E_1.$

Recall that $P(F_1) = \cl(q^{-1}(F_1))$ is a disk which $q$ maps onto the disk $\cl(F_1).$  Let $E^* = P(F_1) \cap q^{-1}(E).$  Then $E^*$ is a component of $q^{-1}(E)$ and of $P(F_1) \cap \partial \Delta$ which $q$ maps onto $E_1.$  Moreover, there is a unique point $v^* \in \inr(E^*)$ such that $q(v^*) = v$ and $q$ maps each component of $E^* - \{v^*\}$ homeomorphically onto $E_1 - \{v\}.$  Further recall that $q$ maps $P(F_1) - E^*$ homeomorphically onto $\cl(F_1) - E_1$ and $q$ \emph{zips up} $E^*$ so that $q(\inr(E^*)) = \inr(E_1) \cup \{v\}$ is a subset of $\inr(\cl(F_1)) = \inr(F_1).$  

In this situation, we chose $w_E = v.$  Let $S_1$ be the closure of the component of $S - \{w_E\}$ that contains $w_{F_1}.$  Let $A_1 = \cup \{ \cl(F) : F \in \mcal{F}$ and $w_F$ is a vertex of $S_1 \}.$  Then $S_1$ is a tree and $w_E$ is an endpoint of $S_1.$  Let $T_1$ be a maximal subtree of $S_1$ that contains $w_E$ and has order 2 at every other vertex of $T_1$ of the form $w_{E'}$ where $E' \in \mcal{E}.$  Let $G_1 = \cup \{ \cl(F) : F \in \mcal{F}$ and $w_F$ is a vertex of $T_1 \}.$  Since $G_1$ is a finite union of disks that meet in boundary arcs of each and have empty triple intersections, then $G_1$ is a compact 2-manifold.  The strong deformation retraction of $A$ onto $S$ restricts to a strong deformation retraction of $A_1$ onto $S_1$ which, in turn, restricts to a strong deformation retraction of $G_1$ onto $T_1.$  Thus, $G_1$ is contractible and, hence, is a disk.

	$w_E$ is an endpoint of the tree $T_1,$ and the unique edge of $T_1$ containing $w_E$ has its other endpoint at $w_{F_1}.$  Hence, $\cl(F_1)$ is the only element of $\{ \cl(F) : F \in \mcal{F}$ and $w_F$ is a vertex of $T_1 \}$ that intersects $E.$  Since $\cl(F_1) \cap E = E_1,$ then $G_1 \cap E = E_1.$  Thus, $\inr(E_1) \cup \{v\} \subset \inr(F_1) \subset \inr(G_1).$  Hence, $b \in \inr(E_1) \subset \inr(G_1).$  In this case, $\partial G_1 = \cup \{ \cl_D(\partial F) : F \in \mcal{F}$ and $w_F$ is a vertex of $T_1 \}.$   Hence, as in the previous case, $\partial G_1 \subset \fr(A) \subset B.$  


	Choose an arc $E'$ in $\inr(E_1)$ such that $b \in \inr(E').$  Then $v \notin E'.$  For $i = 1, 2,$ choose disks $G_i' \subset \inr(F_1)$ such that $G_i' \cap E_1 = E' \subset \partial G_i'$ and $G_1' \cap G_2' = E'.$  Let $G' = G_1' \cup G_2'.$  Then $G'$ is a disk in $\inr(F_1)$ and, hence, in $\inr(G_1)$ such that $v \notin G'$ and $b \in \inr(G').$  Let $j : \partial G'\to D - \{b\}$ denote the inclusion map.  Since $G_1 - \inr(G')$ is an annulus in $D - \{b\}$ with boundary components $\partial G'$ and $\partial G_1,$ then $j$ is homotopic in $D - \{b\}$ to a map $k : \partial G'\to D - \{b\}$ such that $k(\partial G') \subset \partial G_1.$  Hence, the induced maps $j_* : H_1(\partial G')\to H_1(D - \{b\})$ and $k_* : H_1(\partial G')\to H_1(D - \{b\})$ are equal.  Since $\partial G_1 \subset B \subset D - \{b\},$ then $k_*$ factors through the finite group $H_1(B).$  Since $H_1(\partial G') \approx Z,$ it follows that $k_*$ and, hence, $j_*$ are not monomorphisms, contradicting Lemma \ref{mono on homology}.  This concludes the proof of the Theorem in Case ii). \qed \\

\textbf{References}\\

[AGS]  F. D. Ancel, C. R. Guilbault, P. Sparks, in preparation.

[Bo]  E. Borghini, Some remarks on PL collapsible covers of 2-dimensional polyhedra, arXiv:1802.01538. 

[Br]  J. L. Bryant, Piecewise linear topology, pages  219-259 in Handbook of Geometric Topology, edited by R. J Daverman and R. B. Sher, Elsevier, 2002.

[G]  D. Gabai, The Whitehead manifold is the union of two Euclidean spaces, J. Topol. 4 (2011), 529-534.

[GRW]  D. J. Garity, D. Repovs, D. G. Wright, Contractible 3-manifolds and the double 3-space property, to appear in Trans. Amer. Math. Soc.

[GST]  B. Green, N. A. Scoville, M. Tsuruga, Estimating the discrete geometric Lusternik-Schnirelmann category, Topol. Methods Nonlinear Anal. 45 (2015), 103-116.

[H]  A. Hatcher, Algebraic Topology, Cambridge University Press, 2002. 

[M]  B. Mazur, A note on some contractible 4-manifolds, Annals of Math. 73 (1961), 221-228.

[S]  P. Sparks, Contractible $n$-manifolds and the double $n$-space property, arXiv:1502.03154 (2015).

[W]  J. H. C. Whitehead, A certain open manifold whose group is unity, Quart. J. Math 6 (1935), 268-279.

[Z]  E. C. Zeeman, On the dunce hat, Topology 2 (1964),

\end{document}